\documentclass{amsart}[35pt]
\usepackage{amsmath, amsthm, amscd, amssymb, amsfonts, inputenc, tikz}
\UseRawInputEncoding
\DeclareMathOperator{\sgn}{sgn}
\usetikzlibrary{shapes,arrows}

\usepackage[margin=3cm]{geometry}

\newtheorem{theorem}{Theorem}[section]
\newtheorem{lem}[theorem]{Lemma}
\newtheorem{prop}[theorem]{Proposition}
\newtheorem{cor}[theorem]{Corollary}
\theoremstyle{definition}
\newtheorem{defn}[theorem]{Definition}
\newtheorem{ex}[theorem]{Example}

\theoremstyle{remark}
\newtheorem{rem}[theorem]{Remark}
\numberwithin{equation}{section}
\begin{document}
\title[Mean Value and comparative convex functions]
{Mean Value Functions and comparative convexity}
\author{M.H.Hooshmand}
\address{
Department of Mathematics, Shiraz Branch, Islamic Azad University, Shiraz, Iran}

\email{\tt hadi.hooshmand@gmail.com, hooshmand@iaushiraz.ac.ir}


\date{}
----------------------------------------------
\begin{abstract}
During the study of the topic of limit summability of functions (introduced by the author in
2001), we encountered some types of functions that are related to the mean value theorem. In this
paper, we formally define mean value and pointwise MV-functions associated with a given real function
and extend some aspects of the mean value theorem and properties. Also, we introduce and study
an induced conception which we call comparative convexity (and concavity). As many applications
of the study, we prove some uniqueness conditions for related functional equations and completely
solve several functional inequalities.
\end{abstract}
\maketitle
\section{Introduction}
When we were generalizing Euler type constants using limit summability of real functions (see \cite{MHMH,MH}), we found that we can use some classes of useful and applicable functions which are induced from  the mean value theorem (MVT). Thereafter, by searching through some relevant sources, we found several articles in which the related functional equations were studied (e.g., see \cite{Acz, AczKuch, AczKuch2, ZolBa, Mat, SaRi}). Up to our observation, they have studied mean value types theorems, generalizations, and associated functional equations, but in non of them the related functions were explicitly defined and studied. During the study of such functions, we also obtained an induced conception which  we call comparative convexity (and concavity).
\section{Mean Value functions of a real function}
 Now, let us formally define MV-functions as follows. During this paper we consider $I$
as a (bounded or unbounded) interval.
\begin{defn}
Let $I$ be an interval and $f:I\rightarrow \mathbb{R}$ a given function. We call
$g:I^\circ\rightarrow \mathbb{R}$ a ``Mean Value function of $f$" (MV-function) if for every two distinct elements
$x,y\in I$ there exists a $c$ ($=c_{x,y}$) between them such that $f(y)-f(x)=g(c)(y-x)$.
\end{defn}
If $f$ is differentiable on $I^\circ$ and continuous on $I$, then $g=f'$ is
a MV-function of $f$. The next example shows possibility of existence of other MV-functions even
if $f$ is (infinitely) differentiable every where.
\begin{ex}
The Dirichlet function $D=\chi_\mathbb{Q}$ (that is the indicator of rational numbers)
is a MV-function for the identity and all constant functions. Indeed,
if $f$ is a constant (resp. the identity) function, then every function $g$ such that $g^{-1}(0)$
(resp. $g^{-1}(1)$) is a dense subset of $\mathbb{R}$ (e.g., $g=\chi_\mathbb{Q}$
or $g=1-\chi_\mathbb{Q}$) is a MV-function of $f$.
\end{ex}
Now, two natural questions arise that: how large is the class of MV-functions of $f$?, and under what conditions
$f'$ is the only MV-function?

\begin{theorem}
Every function $f:I\rightarrow \mathbb{R}$ has infinitely many MV-functions, and cardinality of the class of
such functions is $2^\mathfrak{\mathfrak{c}}$ $($the biggest possible cardinality, see the next remark$)$.\\
But, if $f$ is differentiable and $f'$  $($resp. $f'^{-1})$ is continuous,
then $f'$ is the unique continuous $($inverse-continuous$)$ MV-function of $f$.
\end{theorem}
\begin{proof}
Let $g:I\rightarrow \mathbb{R}$ be a function
such that preimage of every point (i.e., $g^{-1}(\{ t\})$; $t\in \mathbb{R}$) is dense in $\overline{I}$.
For each distinct  $ x,y\in \mathbb{R}$, we can pick an arbitrary point
$t\in g^{-1}\left(\frac{f(y)-f(x)}{y-x}\right)\cap (\min\{x,y\},\max\{x,y\})$ and put
$c=c_{x,y}:=t$ in the definition. Now, we show that the cardinality of such functions $g$ is $2^c$.
A function $g$  has the property if and only if it is everywhere surjective
(i.e., it transforms every open interval into the whole real line).
For if $g$ is everywhere surjective, then given $x\in \mathbb{R}$ and $(x-\varepsilon,x+\varepsilon)\subseteq I$, we have
 $f((x-\varepsilon,x+\varepsilon))=\mathbb{R}$, thus $f^{-1}(t)\in (x-\varepsilon,x+\varepsilon)$ for all $t\in \mathbb{R}$.
 Since this must hold for every $\varepsilon>0$ such that  $(x-\varepsilon,x+\varepsilon)\subseteq I$, we get the property.
Conversely, if the preimage $g^{-1}(t)$ is dense in $\overline{I}$, for every $t$, then it must intersect every open set.
 In particular, it must intersect every set of the form $(x-\varepsilon,x+\varepsilon)$ in $I$,
hence $g$ is everywhere surjective. Therefore, Theorem 4.3 of \cite{Gur} completes the first part of proof.\\
For the second part, let $f$ be differentiable and $g$ be a MV-function of $f$.
If $g$ is continuous, then one can easily see that $g=f'$. But, if $g$ is invertible
and $g^{-1}$ is continuous, then put $h=g^{-1}$ and consider $x<y$  (without lose of generality).
Thus $x<h(\frac{f(y)-f(x)}{y-x})<y$. By fixing $x$ and $y\rightarrow x$
we get
$$
\lim_{y\rightarrow x^+}h(\frac{f(y)-f(x)}{y-x})=x.
$$
This implies $hf'_+(x)=x$, since $h$ is continuous. Hence $h=f'^{^{-1}}$ and so $g=f'$.
\end{proof}
\begin{rem}
It is important to know that the family of MV-functions of $f:I\rightarrow \mathbb{R}$ is
$2^\mathfrak{c}$-lineable as the aspect of pathological properties of functions (see \cite{Far, Gur}).
Also, one can construct infinitely many such functions $g:\mathbb{R}\rightarrow \mathbb{R}$ in the above theorem (that preimage of every point is dense)
explicitly. A famous example is Conway's base 13 function (see \cite{Rad}).
Another method to construct such functions is possible as follows:
consider the additive group $\mathbb{R}/\mathbb{Q}$. It is clear that this group has
cardinality $\mathfrak{c}$, so there is a bijection $\varphi: \mathbb{R}/\mathbb{Q}\to \mathbb{R}$.
Defining $g(x):=\varphi([x]_{\mathbb{R}/\mathbb{Q}})$ gives us a function with the desired property.
\end{rem}
As we know behaviors of the derivative of a function determines several properties
of that function. Thus it is very important to know that when the role of a MV-function
$g$ is similar to the derivative function.
\begin{lem}$($Basic relations between a real function and its MV-functions$)$\\
Let $g$ be a MV-function of $f:I\rightarrow \mathbb{R}$.\\
(a) If $g\geq 0$ $($on $I^\circ)$, then $f$ is increasing. But the converse is not true $($although it is true when $g=f')$.\\
(b) If $g$ is continuous, then $f$ is differentiable $($but the converse is not true$)$ and $g=f'$ on $I^\circ$.\\
Hence, every function has at most one continuous MV-function and the followings are equivalent:\\
- $f$ is continuously differentiable;\\
- $f$ has a continuous MV-function;\\
- $f'$ is the only continuous MV-function of $f$.\\
(c) If $g$ is increasing, then $f$ is convex. But the converse is not true
$($however, if $g$ is continuous then the converse is also valid$)$. \\
(d) If $g$ is incraesing and $0\in I$, then there exist constants $a,b$ such that
 \begin{equation*}
a+bx\leq f(x)\leq a+xg(x)\;\; ;\;\;x\in I^\circ.
 \end{equation*}
  (e) If $g$ is invertible, then $g^{-1}(\frac{f(y)-f(x)}{y-x})$ is a strict mean function on $I$
$($see the next explanation$)$.
\end{lem}
\begin{proof}
(a) The first part is obvious, and for the second part let $f(t)=t$, and $g(t)$ be $1$ on $\mathbb{Q}$
and $-1$ on $\mathbb{Q}^c$.\\
(b) The first statement is easy to prove, and for disproving the converse consider every differentiable function
$f:\mathbb{R}\rightarrow \mathbb{R}$ such that $g=f'$ is not continuous. Theorem 2.3 proves
the rest of sentences in this part.\\
(c) Let $x<y<z$ be arbitrary elements of $I$ and $c_{x,y}$, $c_{y,z}$ are as mentioned in Definition 2.1.
Then, $x<c_{x,y}<y<c_{y,z}<z$ and $g(c_{x,y})=\frac{f(y)-f(x)}{y-x}$, $g(c_{y,z})=\frac{f(z)-f(y)}{z-y}$.
These relations imply that $f$ is convex if $g$ is increasing (indeed the inequalities $(3.6)$ hold).
The converse is not true by Example 3.10.\\
(d) Putting $a=f(0)$ and $b=g(0)$, one can get the result easily from (c). \\
(e) This is clear by considering the definition.
\end{proof}
\begin{rem}
Up to now, we determined continuous MV-functions of a given function and
will do for increasing (monotone) MV-functions in Theorem 3.9. Hence, we leave the following important problem and questions:\\
{\bf Problem I.} Characterize (or give some necessary and sufficient conditions for existence of) MV-functions
with the intermediate value (Darboux) property, piecewise continuity, left or right continuity, etc.\\
{\bf Question II.} Let $f$ be a real function with one sided derivatives and at most finitely many non-differentiable
points. Then, are there some MV-functions of $f$ with at most finitely many non-continuous
points?. What happens if such points are at most countable? \\
{\bf Question III.} Let $f$ be a non-continuously differentiable function.
Is it true that $f'$ is the only MV-function with the intermediate value property?.
If no, what conditions are needed to be true? \\
\end{rem}

Putting $I\check{\times} I:=I\times I\setminus\{(x,x):x\in I\}$, a function
$\eta:I\check{\times} I\rightarrow I^\circ$ is called ``strict mean function on $I$''
if $\min\{x,y\}<\eta(x,y)<\max\{x,y\}$ , for all $(x,y)\in I\check{\times}I$. Replacing $I\check{\times} I$ by
$I\times I$ and $<$ by $\leq$,  $\eta$ is called ``mean function on $I$''.\\
Let $g$ be a MV-function of $f:I\rightarrow \mathbb{R}$. Then, we can construct functions
 $\eta:I\check{\times} I\rightarrow I^\circ$ such that $\eta(x,y)=c=c_{x,y}$,
 for a value $c$ in the definition (choice functions). Therefore,
 $\eta$ is a strict mean function on $I$ and it can be considered as a symmetric function (i.e. $\eta(x,y)=\eta(y,x)$), and also can be
naturally extended on whole $I\times I$, by defining   $\eta(x,x)=x$ (to be a mean function).

On the other hand, we have the following equivalent conditions:
\begin{equation}
\min\{x,y\}<\eta(x,y)<\max\{x,y\} \Leftrightarrow 0<\frac{\eta(x,y)-x}{y-x}<1 \Leftrightarrow  0<\frac{\eta(x,y)-y}{x-y}<1
\end{equation}
\begin{equation*}
\Leftrightarrow \eta(x,y)=\lambda(x,y)x+(1-\lambda(x,y))y \Leftrightarrow \eta(x,y)=\mu(x,y)y+(1-\mu(x,y))x,
\end{equation*}
where $0<\lambda(x,y)<1$ and $0<\mu(x,y)<1$. Thus,
for every strict mean function $\eta$ there exist functions $\lambda, \mu:I\times I\rightarrow (0,1)$
such that $$\eta(x,y)=y+(x-y)\lambda(x,y))=x+(y-x)\mu(x,y),$$
for all $(x,y)\in I \times I$, and vice versa.
Hence, we arrive at the following proposition.
\begin{prop}
 For a given function $f:I\rightarrow \mathbb{R}$, the following statements are equivalent:\\
 (a) $g$ is a MV-function of $f$.\\
 (b) There exists a $($symmetric$)$ strict mean function $\eta$ on $I$
such that
 \begin{equation*}
 f(y)-f(x)=(y-x)g(\eta(x,y))\;\; ;\;\;(x,y)\in I\times I.
 \end{equation*}
 (c) There exists a $($symmetric$)$ function $\lambda:I\times I\rightarrow (0,1)$
 such that
\begin{equation*}
 f(y)-f(x)=(y-x)g(\lambda(x,y)x+(1-\lambda(x,y))y)\;\; ;\;\;(x,y)\in I \times I
 \end{equation*}
 (d) There exists a $($symmetric$)$ mean function $\eta$ on $I$
such that
  \begin{equation*}
 f(y)-f(x)=(y-x)g(\eta(x,y))\;\; ;\;\;(x,y)\in I \times I.
 \end{equation*}
(e) There exists a $($symmetric$)$ function $\lambda:I\times I\rightarrow (0,1)$
such that
  \begin{equation*}
 f(y)-f(x)=(y-x)g(x+(y-x)\lambda(x,y))\; ;\; x,y\in I, x<y.
 \end{equation*}
 (f) There exists a $($symmetric$)$ function $\lambda:I\times I\rightarrow (0,1)$
 such that
\begin{equation*}
 f(y)-f(x)=(y-x)g(\min\{x,y\}+|y-x|\lambda(x,y))\;\; ;\;\;(x,y)\in I\times I
 \end{equation*}
or equivalently
 \begin{equation*}
 f(y)-f(x)=(y-x)g(\frac{x+y}{2}+(\lambda(x,y)-\frac{1}{2})|y-x|)\;\; ;\;\;(x,y)\in I \times I.
 \end{equation*}
\end{prop}
\subsection{Pointwise MV-functions}
It may be helpful studding the next easier case of MV-functions.
\begin{defn}
Let $f:I\rightarrow \mathbb{R}$ be a given function and fix an element $x_0\in I$. We call
$g:I^\circ\rightarrow \mathbb{R}$ a ``Mean Value function of $f$ at $x_0$" ($x_0$-MV-function) if one of the equivalent conditions (a)-(f)
(in Proposition 2.7) holds, for every two distinct elements $x,y\in I$ with $x=x_0$.
\end{defn}
It is obvious that $g$ is a MV-function of $f$ if and only if it is a $x_0$-MV-function for all $x_0\in I$.
\begin{lem}
All functions of the form
\begin{equation}
g(t)=g_{\mu,x_0}(t):=\frac{\mu(t)}{t-x_0}(f(x_0+\frac{t-x_0}{\mu(t)})-f(x_0)),
 \end{equation}
$($for all $t$ such that $\frac{t-x_0}{\mu(t)}\in I-x_0)$  are $x_0$-MV-functions of $f$,
where $\mu:I\rightarrow (0,1)$ satisfies the functional equation
\begin{equation}
 \mu(t\mu(t)-x_0\mu(t)+x_0)=\mu(t)\;\; ;\;\;t\in I.
 \end{equation}
In particular, putting  $\mu(t)=\mu$ where $\mu\in(0,1)$,
we obtain infinitely many such functions,
and specially $g_{\frac{1}{2},x_0}(t)=\frac{f(2t-x_0)-f(x_0)}{2t-2x_0}$, if $\mu=\frac{1}{2}$.
\end{lem}
\begin{proof}
For every $t\in I\setminus \{x_0\}$ put $z=z_{t,x_0}=\mu(t)t+(1-\mu(t))x_0$, then
$$
\frac{z-x_0}{\mu(t)}=t-x_0\in I-x_0,
$$
so $g(z)$ is defined and we have $g(z)=\frac{f(t)-f(x_0)}{t-x_0}$
\end{proof}
A question arises that when a $x_0$-MV-function is the same $f'$. For the case
$x_0=0$ and $\mu(t)=\mu$ it means that
\begin{equation*}
f'(t)=\frac{\mu}{t}(f(\frac{t}{\mu})-f(0))\;\; ;\;\;0\neq t\in \mu I,
 \end{equation*}
and we arrive at the differential equation
\begin{equation}
tf'(t)-\mu f(\frac{t}{\mu})+\mu f(0)=0\;\; ;\;\;t\in \mu I.
 \end{equation}
The related integral equation is
 \begin{equation}
\int^x\frac{f(\frac{1}{\mu}t)}{t}dt=\frac{1}{\mu}f(x)+f(0)\log|x|\; ; \; x\neq 0.
 \end{equation}
that is interesting and may be useful for the topic. It is obvious that all functions $f(t)=kt$ enjoy the properties.
\section{Comparative convex and concave functions}
We defined the MV-functions via replacing $f'$ by an arbitrary function $g$ satisfying the equation.
On the other hand a function $f$ is convex on $I$ if and only if $f(x)\geq f(y)+f'(y)(x-y)$, for all
$x,y \in I$. Therefore, we are led to the following definition.
\begin{defn}
Let $g:I\rightarrow \mathbb{R}$ be a given function. We call
$f:I\rightarrow \mathbb{R}$
``$g$-compared convex'' if $f(x)\geq f(y)+g(y)(x-y)$, for all $x,y \in I$.
Here $g$ is called ``difference quotient bound'' of $f$ (see the inequality $(3.6)$).
 Analogously, ``strictly $g$-compared convex'' and ``$g$-compared concave'' functions are defined.
\end{defn}
As a geometric interpretation, $f$ is $g$-compared convex if and only if for every point
$x_0\in I$ there is a line of support of $f$ with the slope $g(x_0)$.
If $f$ is a differentiable convex function, then it is $f'$-compared convex. More generally, every
convex function is $f'_{\pm}$-compared convex. It is clear that $f$ is $g$-compared concave if and only if
$-f$ is  $-g$-compared convex. Also, $f$ is $0$-compared convex if and only if
$f$ is constant.\\
{\bf Convention.} For comfort, in this paper,  we use ``$g$-convex'' instead of ``$g$-compared convex''.
\begin{ex}
The absolute value function is $\sgn$-convex. Indeed $|.|$ is $g$-convex if and only if
there exists $-1\leq \alpha\leq 1$ such that
\begin{equation}
g(x) = \begin{cases}
             \sgn(x)  & \text{if } x\neq 0 \\
             \alpha  & \text{if } x=0
       \end{cases}
 \end{equation}

Conversely, $f$ is $\sgn$-convex if and only if $f=|.|+c$ for some constant $c$ (also see Theorem 3.8) .
\end{ex}
\begin{prop}
 A function $f:I\rightarrow \mathbb{R}$ is $g$-convex if and only if one of
 the following equivalent conditions hold:\\
\begin{equation}
 g(x)(y-x)\leq f(y)-f(x)\leq g(y)(y-x)\;\; ;\;\; x,y \in I.
 \end{equation}
\begin{equation}
f(\lambda x+(1-\lambda)y)\leq f(y)+\lambda g(x)(x-y)\;\; ;\;\; x,y \in I,\; \lambda \in [0,1].
 \end{equation}
\begin{equation}
 g(x)\leq \frac{f(y)-f(x)}{y-x}\leq g(y)\;\; ;\;\; x,y \in I,\; x<y.
 \end{equation}
\begin{equation}
f(\lambda x+(1-\lambda)y)\leq f(x)+(1-\lambda)g(y)(y-x)\;\; ;\;\; x,y \in I,\; \lambda \in [0,1].
 \end{equation}
  \end{prop}
  \begin{proof}
If $f$ is $g$-convex, then by changing the roles of $x$, $y$ in its inequality, and combining the two inequalities
we arrive at $(3.2)$. Now, let $x<y$ in $I$ and consider $0<\lambda<1$. If $(3.2)$  holds, then
$$
g(x)\leq \frac{f(\lambda x+(1-\lambda)y)-f(x)}{(1-\lambda)(y-x)}\leq g(\lambda x+(1-\lambda)y)\leq
\frac{f(y)-f(\lambda x+(1-\lambda)y)}{\lambda(y-x)}\leq g(y)
$$
Thus, we obtain two inequalities $$f(\lambda x+(1-\lambda)y)\leq f(y)+\lambda g(x)(x-y)$$ and
$$f(\lambda x+(1-\lambda)y)\leq f(x)+(1-\lambda)g(y)(y-x),$$ for $x<y$ and $0<\lambda<1$. By changing the roles
of $x$, $y$ and replacing $1-\lambda$ by $\lambda$ in the second inequality, we drive the first inequality
for $y<x$ and all $0<\lambda<1$. Also, $(3.3)$ holds for $\lambda=0,1$ obviously. Hence, we arrive at $(3.3)$
for all $x,y \in I$ and $\lambda \in [0,1]$. \\
Now, if $(3.3)$ holds, then putting $\lambda=1$ we conclude that $f$ is $g$-convex and so $(3.2)$
is valid, thus we deduce $(3.4)$. Notice that $(3.4)$$\Rightarrow$$(3.5)$ is similar to $(3.2)$$\Rightarrow$$(3.3)$, and
$(3.5)$ implies $g$-convexity of $f$ by putting $\lambda=0$. Therefore, the proof is complete.
\end{proof}
\begin{rem}
It is easy to show that if $I$ is open or $f$ is continuous, then we can reduce the condition $\lambda \in [0,1]$
to $\lambda \in (0,1)$. Now, the question is that wether it is valid without the extra conditions?, i.e.,
does $f(\lambda x+(1-\lambda)y)\leq f(x)+(1-\lambda)g(y)(y-x)$ for $0<\lambda<1$ (on $I$) imply that
$f$ is $g$-convex?
\end{rem}
An important result of Proposition 3.3 is
\begin{equation}
 g(x)\leq \frac{f(y)-f(x)}{y-x}\leq g(y)\leq \frac{f(z)-f(y)}{z-y}\leq g(z),
\end{equation}
for all $x<y<z$ in $I$. Hence, we arrive at the following important corollary.
\begin{cor}
A function $f$ is convex if and only if
it is $g$-convex for some $g$.
Also, if $f$ is $g$-convex, the $f$ is convex and $g$ is increasing.
\end{cor}
Now, we are ready to prove the essential relations between $f$ and $g$ for $g$-convex functions $f$. It together with
Proposition 3.3 enable us to characterize them in several ways.

\begin{theorem}
If $f:I\rightarrow \mathbb{R}$ is $g$-convex, then\\
\begin{equation*}
f'_-(x)=g(x^-)\leq g(x)\leq g(x^+)=f'_+(x)\;\; ;\;\; x \in I^\circ.
 \end{equation*}
Hence, $f$ is differentiable (on $I^\circ$) if and only if $g$ is continuous, and if this is the case,
then $g=f'$ on $I^\circ$.
\end{theorem}
\begin{proof}
The inequality $(3.6)$ together with increasingness of $g$ imply that  $f'_-(x)\leq g(x^-)\leq g(x)\leq g(x^+)\leq f'_+(x)$.
But we claim that $f'_+(x)=g(x^+)$ and $f'_-(x)=g(x^-)$.
For fixed $x < y$ in $I^\circ$ consider the function
$$
 h(t) = f(t) - \frac{f(y)-f(x)}{y-x}(t-x)
$$
which has the left derivative
$$
h'_-(t) = f'_-(t) - \frac{f(y)-f(x)}{y-x} \, .
$$
We have $h(x) = h(y)$ so that $h$ attains its minimum on the interval $[x, y]$ at some point $t_0$ with $x < t_0 \le y$. Then $h'_-(t_0) \le 0$ so that
$$
 \frac{f(y)-f(x)}{y-x} = f'_-(t_0) - h'_-(t_0) \ge  f'_-(t_0) \ge g(t_0^-) \ge g(x^+)  \, .
$$
Taking the limit $y \to x^+$ we conclude that $f'_+(x) \ge g(x^+)$. This proves the first equality, the second one can be proved in the same way.
\end{proof}
The first corollary states a necessary and sufficient condition on $(f,g)$ for $g$-convexity of $f$.
\begin{cor}
A function $f:I\rightarrow \mathbb{R}$ is $g$-convex if and only if it is
convex and $f'_-\leq g\leq f'_+$ on $I^\circ$.
Hence, every differentiable convex function has a unique difference quotient bound function (that is $f'$).
\end{cor}
Now, we are ready to state some characterizations of $g$-convex and their difference quotient bound functions.
\begin{theorem} Assume that the interval $I$ is open.\\
(a) $($A characterization of $g$-convex functions$)$.
Let $g:I\rightarrow \mathbb{R}$ be a given increasing function. Then,
$f:I\rightarrow \mathbb{R}$ is $g$-convex if and only if there is a constant $c\in I$
such that for all $x\in I$,
\begin{equation}
f(x)-f(c)=\int_c^x g(t)dt.
 \end{equation}
(b) $($A characterization of difference quotient bounds of $f)$.\\
Let $f:I\rightarrow \mathbb{R}$ be a given convex function. Then,
$g:I\rightarrow \mathbb{R}$ is a difference quotient bound of $f$ if and only if there exists
$J\subseteq I$ such that $f$ is differentiable on $J$,  $I\setminus J$ is at most countable
$($thus $I\setminus J=\{c_n\}_{n=1}^N$ or $\emptyset$, for some $N\in \mathbb{N})$
and
\begin{equation}
g(x) = \begin{cases}
             f'(x)  & \text{if } x\in J \\
             d_n  & \text{if } x=c_n\in I\setminus J
       \end{cases}
 \end{equation}
 for some $d_n\in [f'_-(c_n),f'_+(c_n)]$.
\end{theorem}
\begin{proof}
The above corollary implies that ($f$ is convex and so) $g=f'$ on $I$
except possibly in at most countable points
(i.e., there exists
$J\subseteq I$ such that $f$ is differentiable on $J$,  $I\setminus J$ is at most countable
and thus $I\setminus J=\{c_n\}_{n=1}^N$ or $\emptyset$, for some $N\in \mathbb{N}$).
Also, we have $f'_-\leq g\leq f'_+$ on $I$. Therefore,
there exists $c\in I$ such that
$$
f(x)-f(c)=\int_c^x f'_+(t)dt=\int_c^x f'_-(t)dt=\int_c^x g(t)dt.
$$
Now, one can complete the proof easily.
\end{proof}
\subsection{Relations with MV-functions}
The following theorem says that if a function has an increasing MV-function $g$,
then it is $g$ convex, and characterizes such functions $g$. Hence, it can be considered as a
generalization of a famous fact about real convex functions.
\begin{theorem}
If $f:I\rightarrow \mathbb{R}$ has an increasing MV-function $g$, then it
is $g$-convex $($but the converse is not true$)$, hence $f'_-(x)=g(x^-)\leq g(x)\leq g(x^+)=f'_+(x)$
$($on $x \in I^\circ)$ and the inequalities  $(3.2)$ till $(3.6)$ hold.
\end{theorem}
\begin{proof}
This is a direct result of Lemma 2.5(c), Theorem 3.6 and the next example.
\end{proof}
Now, we show that there exists a convex function without any increasing MV-function.
\begin{ex}
The real function $f(x)=|x|$ dose not have any increasing (monoton) MV-function
(although it is convex). For if $g$ is an increasing MV-function of $f$, then $g$
is of the form $(3.1)$ in Example 3.2 that can not satisfy one of the equivalent conditions
in Proposition 2.7 (because the equation $|y|-|x|=g(\eta(x,y))(y-x)$ is impossible for all $x<0<y$).
Note that, here, $f'_{\pm}$ are increasing but not MV-functions of $f$.
\end{ex}
Considering the above example, we arrive at the following problem.\\
{\bf Problem IV.} Give an example of a function $f$ with more than one increasing
MV-function ($f$ must be convex and non-differentiable at some points).
 \section{Mean value and comparative convex functional equations and inequalities}
Now, we focus on functional equations and inequalities induced by the topics MV-functions
and  compared convexity.
\subsection{MV-functional equations and inequalities}
 The equivalent conditions mentioned in Propositions 2.7 and 3.3 give us several
 equivalent functional equations and inequalities. Now,
 we discuss them.\\
Let $I$ be an interval, $f:I\rightarrow \mathbb{R}$ and
$g:I^\circ\rightarrow \mathbb{R}$ real functions, and $\eta:I\times I\rightarrow I^\circ$
a strict mean function,
$\lambda:I\times I\rightarrow (0,1)$ and $h:D_h\rightarrow \mathbb{R}$
 (where $D_h$ contains the image of the function $\frac{f(y)-f(x)}{y-x}$ for all $(x,y)\in I\check{\times}I$).
Note that some of these functions are unknown and others known functions.
Then, we call $(4.1)$ (resp. $(4.2)$) MV-functional equation (resp. MV-functional
inequality).
 \begin{equation}
 \frac{f(y)-f(x)}{y-x}=g(\eta(x,y))\;\; ;\;\;(x,y)\in I\check{\times}I
 \end{equation}
\begin{equation}
 0<\frac{1}{y-x}(h(\frac{f(y)-f(x)}{y-x})-x))<1\;\; ;\;\;(x,y)\in I\check{\times}I
\end{equation}
It is interesting to know that the MVT-equation and inequality have close relations if $h=g^{-1}$.
In this case, every solution $(f,g)$ of $(4.1)$ satisfies $(4.2)$. But the converse is true if
$g$ is monotone and enjoys the Darboux property. For if it is so, then  $(4.2)$ implies that
$\frac{f(y)-f(x)}{y-x}$ (for $(x,y)\in I\check{\times}I$) lies between $x$ and $y$, and so $g$ is
a monotonically continuous MV-function of $f$, hence $(f,g)$ satisfies  $(4.1)$
(for some  strict mean function $\eta$). \\
Note that $(2.1)$ provides several equivalent inequalities to $(4.2)$
that one of them is
\begin{equation*}
 \min\{x,y\}<h(\frac{f(y)-f(x)}{y-x})<\max\{x,y\}\;\; ;\;\;(x,y)\in I\check{\times}I
\end{equation*}
For studying MV-equations we divided them to the following types:\\
{\bf (a) $f$ is a known function.} \\
The MV-equation has infinitely many solutions $(g,\eta)$. For if $\eta:I\check{\times} I\rightarrow I$
is an arbitrary invertible strict mean, then for every $t\in I^\circ$
there exists a unique $(x_t,y_t)\in I\check{\times} I$ such that $t=\eta(x_t,y_t)$. Now, defining
$g=g_\eta$ (on $I^\circ$)  by $g(t):=\frac{f(y_t)-f(x_t)}{y_t-x_t}$ we observe that $(g,\eta)$ satisfies
the equation. Since the family of such functions $\eta$ is infinite, we are done. Also, the MV-inequality
has infinitely many solutions $h$, by Theorem 2.3 and the above explanation.
Therefore, a uniqueness theorem should be useful for this case.
\begin{lem} $($A uniqueness conditions for the MV-equation and inequality$)$.
If $f$ is differentiable and $f'$ is invertible, then $h=f'^{-1}$
is a solution of the MV-inequality. Moreover, if $f'^{^{-1}}$ is continuous $($e.g., if $f'$ is strictly
increasing and continuous$)$  then  $h=f'^{^{-1}}$
$($resp. $g=f')$ is the unique solution of the MV-inequality $($resp. MV-equation$)$ with the properties.\\
\end{lem}
\begin{proof}
 The first part of the claim is clear by Lemma 2.5(e) and $(2.1)$. Now, let $h$ be an arbitrary solution of the inequality and
$x<y$ (without lose of generality). Then, $x<h(\frac{f(y)-f(x)}{y-x})<y$. By fixing $x$ and $y\rightarrow x$
we get
$$
\lim_{y\rightarrow x^+}h(\frac{f(y)-f(x)}{y-x})=x.
$$
This implies $hf'_+(x)=x$, since $h$ is continuous. Thus the assumption yields $h=f'^{^{-1}}$.
Therefore, $g=h^{-1}=f'$ is also a unique solution of the MV-equation.\\
{\bf (b) $g$ (resp. $h$) is a known function.}
\end{proof}
For this case, we have also infinitely many solutions $(f,\eta)$ (resp. $f$) for
$(4.1)$ (resp. $(4.2)$), e.g., put $\eta(x,y)=c$ and $f(t)=g(c)t$, for all
constants $c\in I$.
However, the next theorem gives the general solution of them under some conditions on $g$ and $h$.
\begin{theorem} If  $h$ $($resp. $g)$ is strictly monoton and continuous, then all solutions
of the MV-inequality $($resp. MV-equation$)$ are of the form
$$
f(t)=\int^t h^{-1}(s)ds+c,
$$
(resp.
$$
f(t)=\int^t g(s)ds+c,
$$
$$
\eta(x,y)=g^{-1}(\frac{f(y)-f(x)}{y-x}),
$$
of curse, $\min\{x,y\}<\eta(x,y)<\max\{x,y\}$ $)$, for all real constants $c$.
\end{theorem}
\begin{proof} If  $f(t)=\int^t h^{-1}(s)ds+c$ and $x<y$ then there exists a $x<z<y$ suxh that
$$
\frac{f(y)-f(x)}{y-x}=\frac{1}{y-x}\int_x^y h^{-1}(s)ds=h^{-1}(z),
$$
and so the inequality $(2.2)$ is equivalent to $0<\frac{z-x}{y-x}<1$ that is clearly true.
For uniqueness, if $h$ is strictly monotonic and continiuous, then
$$
h^{-1}(x)<\frac{f(y)-f(x)}{y-x}<h^{-1}(y),
$$
whenever $x<y$. Now, fixing $x$ and $y\rightarrow x^+$ we get $h^{-1}(x)=f'_+(x)$,
and in a similar way $h^{-1}(y)=f'_-(y)$. Therefore, $f'=h^{-1}$ on the related interval,
and we are done it for the MV-inequality. Proof for the MV-equation is similar.
\end{proof}
{\bf (c) $\eta$ is a known function.} \\
In this case, the MV-equation has infinitely many solutions, obviously. As a special interesting
equation we have
$$
\frac{f(y)-f(x)}{y-x}=g(\lambda x+\mu y),
$$
where $\lambda,\mu \geq 0$ and $\lambda+\mu =1$ ($\eta(x,y)=\lambda x+\mu y$). One can see
its general solution in \cite{SaRi}.\\
{\bf (d) Two of the functions $f$, $g$ and $\eta$ are known.} \\
The MV-equation does not have a solution necessarily. For
if $\eta(x,y)=(1-\lambda)x+\lambda y$ where $0<\lambda<1$ is a constant, then
$f$ must be a polynomial of order at most two (see, for example, \cite{SaRi}).  Also, if $g\equiv 0$ and $f$ is not one-to-one,
then they do not satisfy the equation.
\subsection{Compared convex and concave functional inequalities}
Now, we consider the functional inequalities associated to the $g$-convexity (and concavity).
The obtained results of the topic enable us to solve several such functional inequalities.
Hence, we first state their general solutions as a theorem, and then completely solve
several functional inequalities and related systems, as some applications of the study.
The following theorem is a direct result of Theorem 3.8 and Proposition 3.3.
\begin{theorem}
The $(g$-comparative convex$)$ functional inequality
\begin{equation}
 f(x)\geq f(y)+g(y)(x-y)\;\; ;\;\; x,y \in I.
 \end{equation}
has a solution if and only if $f$ is convex, $g$
 is increasing and $f'_+(x)=g(x^+)$, $f'_-(x)=g(x^-)$ on $I^\circ$. Moreover, \\
(a) If $g$ is known and $I$ open, then the general solution is
 \begin{equation}
f(x)-f(c)=\int_c^x g(t)dt=\int_c^x g^{\pm}(t)dt\;\; ;\;\; x\in I,
 \end{equation}
 where $c$ is an arbitrary fixed point in $I$.\\
(b) If $f$ is given, then the general solution is
$$
g(x) =\lambda(x)f'_-(x)+(1-\lambda(x))f'_+(x)\;\; ;\;\; x\in I^\circ,
$$
for all functions $\lambda:I^\circ \rightarrow [0,1]$. \\
Another general solution is of the form
  \begin{equation}
g(x) = \begin{cases}
             f'(x)  & \text{if } x\in I_f \\
             d_n  & \text{if } x=c_n\in I\setminus I_f
       \end{cases}
 \end{equation}
where $I$ is open, $I_f:=\{x\in I: f\mbox{ is differentiable at } x\}$, $I\setminus I_f=\{c_n\}_{n=1}^N$ or $\emptyset$,
$d_n\in [f'_-(c_n),f'_+(c_n)]$.
\end{theorem}
\subsection{Some applications of the comparative convexity and concavity}
 General solutions of the following functional inequalities and related systems are obtained from the topic
 of $g$-convexity and $g$-concavity. In continuation, the interval $I$ is assumed open.
 \begin{ex}
General solutions of the functional inequality (self-convexity that is equivalent to
$f$-convexity of $f$):
 \begin{equation}
f(x)\geq f(y)(1+x-y) \; ; \; x,y\in I,
 \end{equation}
is of the form $f(t)=\lambda e^t$ with $\lambda \geq 0$.
Thus, general solution of
$$
\frac{f(x)}{f(y)}\geq 1+x-y \; ; \; x,y\in \mathbb{R}
$$
is $\lambda e^t$, for all $\lambda\neq 0$. \\
Because $f$ must be continuous (since it is convex an $I$ open) and so $f'=f$
(by Theorem 3.6) and $f$ is increasing.
Therefore, the natural exponential function is the only real function
satisfying $(4.6)$ (or the second inequality) with the initial condition $f(0)=1$.\\
A similar inequality and its general solution can be stated for the self-concavity.
As a generalization of $(4.6)$, let $\phi$ be a given continuous function on $I$, $k$ a real constant and
consider the functional inequality
\begin{equation}
f(x)\geq f(y)+(kf+\phi)(y)(x-y) \; ; \; x,y\in I.
 \end{equation}
If it has a solution, then $f$ satisfies the linear differential equation $f'-kf=\phi$ on $I$. Thus
$f$ can be determined trough $(4.7)$ and the condition that $kf+\phi$ must be increasing.
For example, putting $\phi=0$ on $I$, the general solution of $f(x)\geq f(y)+kf(y)(x-y)$ is
of the form $f(t)=\lambda e^{kt}$ with $\lambda,k \geq 0$.
\end{ex}
\begin{ex}
General solution of the system of functional inequalities (symmetric convexity, that
says $f$ is $g$-convex and $g$ is $f$-convex):
 \begin{equation}
\left\{\begin{array}{l}
f(x)\geq f(y)+g(y)(x-y)\\
\qquad\qquad \qquad\qquad \qquad\qquad \; ; \; x\in I\\
g(x)\geq g(y)+f(y)(x-y)
\end{array} \right.
 \end{equation}
 is
 $$f(t)=\lambda_1 e^t+\lambda_2 e^{-t},\; g(t)=\mu_1 e^t+\mu_2 e^{-t},$$
 for all $t\in I$ and $\lambda_1,\mu_1, \lambda_2,\mu_2$ such that $\lambda_1\theta_I\geq |\lambda_2|$
 and  $\mu_1\theta_I\geq |\mu_2|$, where $\theta_I=\inf_{t\in I} e^{2t}$.\\
 Because both $f$ and $g$ must be continuous and so $f'=g$, $g'=f$. The conditions on the coefficients
 come from increasingness of them.
 Note that if $I=\mathbb{R}$, then the general solution is simplified to the form
 $f(t)=\lambda e^t$, $g(t)=\mu e^t$ for $\lambda, \mu\geq 0$ ($t\in \mathbb{R}$).
\end{ex}
\begin{ex}
General solution of the system of functional inequalities (convex-concavity,
that says $f$ is $g$-convex and $h$-concave):
 \begin{equation}
\left\{\begin{array}{l}
f(x)\geq f(y)+g(y)(x-y)\\
\qquad\qquad \qquad\qquad \qquad\qquad \; ; \; x\in I\\
f(x)\leq f(y)+h(y)(x-y)
\end{array} \right.
 \end{equation}
 is $f(t)=at+b$, $g(t)=h(t)=a$ ($t\in I$),
 for all constants $a,b$.\\
 Because $(4.9)$ implies that $f'=g=h$, $g$ must be increasing and $h$ decreasing.
\end{ex}

 \subsection*{Acknowledgement} The author would like to thank the persons who
 answered or sent their feedback and comments to some of my questions (related to this paper)
 in the math.stackexchange website.

\end{document}